\documentclass[11pt]{article}
\usepackage{amssymb,amsfonts}
\usepackage{amsmath,amsthm,amsxtra}

\usepackage{graphicx}
\usepackage{amscd}
\usepackage{ulem}
\usepackage{makeidx}
\usepackage{fancybox}

\usepackage[linktocpage=true, colorlinks=true, linkcolor=blue, citecolor=red, urlcolor=green]{hyperref}

\newcommand{\ccc}{{\mathbf C}}

\newcommand{\nnn}{{\mathbf N}}

\newcommand{\zzz}{{\mathbf Z}}
\renewcommand{\ggg}{{\frak{g}}}
\newcommand{\hhh}{{\frak{h}}}

\newtheorem{thm}{Theorem}[section]
\newtheorem{prop}{Proposition}[section]
\newtheorem{lemma}{Lemma}[section]
\newtheorem{cor}{Corollary}[section]

\newtheorem{rem}{Remark}[section]
\newtheorem{note}{Note}[section]

\numberwithin{equation}{section}

\begin{document}

\title{Vanishing of the quantum reduction of the Deligne 
exceptional series representations of negative integer level}

\author{\footnote{12-4 Karato-Rokkoudai, Kita-ku, Kobe 651-1334, 
Japan, \qquad
wakimoto.minoru.314@m.kyushu-u.ac.jp, \hspace{5mm}
wakimoto@r6.dion.ne.jp 
}{ Minoru Wakimoto}}

\date{\empty}

\maketitle

\begin{center}
Abstract
\end{center}

In this paper we show that, for the Deligne exceptional series 
representations of negative integer level of affine Lie algebras,
the quantum Hamiltonian reduction vanishes except for the cases
where the nilpotent element is conjugate to $e_{-\theta}$ or 
$e_{-\theta_s}$.

\tableofcontents

\section{Introduction}

In this paper we study the vanishing of the space obtained by 
the quantum Hamiltonian reduction of the Deligne exceptional series
representations of negative integer level, 
as a continuation from \cite{KW2024}.
Let $\ggg$ be the non-twisted affinization of a finite-dimensional simple 
Lie algebra $\overline{\ggg}$ with a Cartan subalgebra $\overline{\hhh}$.
Let $\overline{\Delta}$ (resp. $\overline{\Delta}^+$) be the set of all 
(resp. positive all) roots and $\overline{W}$ be the Weyl group of 
$(\overline{\ggg}, \overline{\hhh})$. Let $( \,\ | \,\ )$ be the standard 
bilinear form defined in \cite{K1}, and $\overline{\Delta}_{\ell}$ (resp. 
$\overline{\Delta}_s$) be the set of long (resp. short) roots.
Let $\theta$ (resp. $\theta_s$) denote the highest root (resp. 
highest short root).
For a root $\alpha \in \overline{\Delta}$, let $\overline{\ggg}_{\alpha}$
(resp. $e_{\alpha}$) denote the root space (resp. root vector) of $\alpha$.
The coroot of $\alpha \in \overline{\Delta}$ is defined by
$\alpha^{\vee} := \dfrac{2\alpha}{(\alpha|\alpha)}$, and 
let $\overline{Q}^{\vee} := \sum_{\alpha \in \overline{\Delta}} 
\zzz \alpha^{\vee}$ be the coroot lattice. 
For the basic notations such as $\delta$ and $\Lambda_0$ and so on, 
we follow from \cite{K1}.  The Cartan subalgebra 
$\hhh := \overline{\hhh}+\ccc \Lambda_0+\ccc \delta$ of $\ggg$ 
is identified with its dual space $\hhh^{\ast}$ by the 
non-degenerate inner product $( \,\ | \,\ )$. Define the coordinates 
in $\hhh$ by
$$
(\tau, z, t) \, := \, 2\pi i(-\tau \Lambda_0+z+t\delta)
$$
where $\tau \in \ccc_+ := \{\tau \in \ccc \,\ ; \,\ {\rm Im}(\tau)>0\}$
and $z \in \overline{\hhh}$ and $t \in \ccc$. For an element 
$\lambda \in \hhh$, let $\overline{\lambda}$ denote its 
$\overline{\hhh}$-component. We put $q:=e^{2\pi i\tau}$ as usual.

\medskip

For $\lambda \in \hhh^{\ast}$ and $\alpha \in \overline{\Delta}_+$,
we put 
\begin{subequations}
{\allowdisplaybreaks
\begin{eqnarray}
F^{[\alpha]}_{\lambda+\rho} &:=& 
\frac12 \, q^{\frac{|\lambda+\rho|^2}{2(K+h^{\vee})}}
\sum_{\gamma \in \overline{Q}^{\vee}}
(\alpha|\gamma) \, t_{\gamma}(e^{\lambda+\rho})
\nonumber
\\[1mm]
&=&
\frac12 \sum_{\gamma \in \overline{Q}^{\vee}}
(\alpha|\gamma) \, e^{\overline{\lambda+\rho}+(K+h^{\vee})\gamma} \ 
q^{\frac{1}{2(K+h^{\vee})}|\overline{\lambda+\rho}+(K+h^{\vee})\gamma|^2}
\label{eqn:2024-703a1}
\\[2mm]
A^{[\alpha]}_{\lambda+\rho} &:=& \sum_{w \in \overline{W}}
\varepsilon(w) \, w(F^{[\alpha]}_{\lambda+\rho})
\label{eqn:2024-703a2}
\end{eqnarray}}
\end{subequations}
where $\rho$ is the Weyl vector of $\ggg$ and 
$h^{\vee} = (\rho| \delta)$ is the dual Coxeter number, and 
$t_{\gamma}$ is the linear transformation on $\hhh$ 
defined by the formula (6.5.2) in \cite{K1}. 

For Deligne exceptional series representations, 
in the case when $\lambda$ satisfies the conditions 
(i) $\sim$ (iv) in Theorem 4.1 in \cite{KW2017d}, 
the numerator of the character of the irreducible 
$\ggg$-module $L(\lambda)$ is given by the above 
$A^{[\alpha]}_{\lambda+\rho}$.
We consider the quantum Hamiltonian reduction of these 
modules. Given an $sl_2$-triple $(x,e,f)$ with $[x,e]=e, \, 
[x,f]=-f, \, [e,f]=x$, the numerator of the character of 
the $W(\ggg,x,f)_K$-module obtained from the quantum reduction 
of $\ggg$-module $L(\lambda)$ is 
$A^{[\alpha]}_{\lambda+\rho}
(\tau, \, -\tau x+H, \, \frac{\tau}{2} \, |x|^2)$. 
In \cite{KW2024}, the following theorem is proved:

\medskip



\begin{thm} {\rm \cite{KW2024}} \,\ 
\label{vanish:thm:2024-704a}
Assume that $(\lambda, \alpha)$ is an element in 
$\hhh^{\ast} \times \overline\Delta_+$ satisfying the condition 
\begin{equation}\left\{
\begin{array}{lcccc}
(\lambda+\rho| \, \delta) &= & K+h^{\vee} &\in & \nnn \\[1mm]
(\lambda+\rho| \, \delta-\alpha) &=& 0 & &
\end{array}\right.
\label{vanish:eqn:2024-703b}
\end{equation}
and $(x,e,f)$ is an $sl_2$-triple. 
Choose $\beta \in \overline{\Delta}$ such that 
\begin{equation} \hspace{-32.5mm}
\left\{
\begin{array}{lcc}
(\beta \, | \, x) &\in & \zzz \\[1mm]
(\beta \, | \, \overline{\hhh}^f) &=& \{0\}
\end{array}\right.
\label{vanish:eqn:2024-703c}
\end{equation}

\noindent
Then, 
\begin{enumerate}
\item[{\rm 1)}] \,\ the following formula holds for $H \in \overline{\hhh}^f$:
\begin{equation}
A^{[\alpha]}_{\lambda+\rho}
\Big(\tau, \,\ -\tau x+H, \,\ \frac{\tau}{2} \, |x|^2\Big)
\,\ = \,\ 
\frac14 \, (\beta |x) \sum_{w \in \overline{W}}
\varepsilon(w) \, (\beta^{\vee}| w\alpha) \, 
f_{\lambda,x,w}(\tau,H)
\label{vanish:eqn:2024-703e1}
\end{equation}

\vspace{-3mm}

\noindent
where

\vspace{-6mm}

\begin{equation}
f_{\lambda,x,w}(\tau,H) \, := 
\sum_{\gamma \in \overline{Q}^{\vee}}
e^{2\pi i(\overline{\lambda+\rho}+(K+h^{\vee})\gamma \, | \, w^{-1}H)} \, 
q^{\frac{1}{2(K+h^{\vee})}|\overline{\lambda+\rho}
+(K+h^{\vee})\gamma-(K+h^{\vee})w^{-1}x|^2}
\label{vanish:eqn:2024-703e2}
\end{equation}

\item[{\rm 2)}] \,\ the RHS of \eqref{vanish:eqn:2024-703e1} does not depend
on the choice of $\beta$ satisfying the condition \eqref{vanish:eqn:2024-703c}.
\end{enumerate}
\end{thm}

\medskip

We note that the RHS of \eqref{vanish:eqn:2024-703e2} can be 
written as follows since $(x|H)=0$:
\begin{equation}
f_{\lambda,x,w}(\tau,H) = \hspace{-1mm}
\sum_{\gamma \in \overline{Q}^{\vee}} \hspace{-1mm}
e^{2\pi i(\overline{\lambda+\rho}+(K+h^{\vee})\gamma -(K+h^{\vee})w^{-1}x
\, | \, w^{-1}H)} \, 
q^{\frac{1}{2(K+h^{\vee})}|\overline{\lambda+\rho}
+(K+h^{\vee})\gamma-(K+h^{\vee})w^{-1}x|^2}
\label{vanish:eqn:2024-703e3}
\end{equation}

In \cite{KW2024}, the modular transformation properties of 
$A^{[\alpha]}_{\lambda+\rho}$ and the functions 
%
\begin{subequations}
{\allowdisplaybreaks
\begin{eqnarray}
A^{[\alpha] \, (-)}_{\lambda+\rho}
\Big(\tau, \,\ -\tau x+H, \,\ \frac{\tau}{2} \, |x|^2\Big)
&:=& 
\frac14 \, (\beta |x) \sum_{w \in \overline{W}}
\varepsilon(w) \, (\beta^{\vee}| w\alpha) \, 
f_{\lambda,x,w}^{(-)}(\tau,H)
\label{vanish:eqn:2024-728a1}
\\[2mm]
A^{[\alpha] \, (\ast)}_{\lambda+\rho}
\Big(\tau, \,\ -\tau x+H, \,\ \frac{\tau}{2} \, |x|^2\Big)
&:=& 
\frac14 \, (\beta |x) \sum_{w \in \overline{W}}
\varepsilon(w) \, (\beta^{\vee}| w\alpha) \, 
f_{\lambda,x,w}^{(\ast)}(\tau,H)
\label{vanish:eqn:2024-728a2}
\end{eqnarray}}
\end{subequations}
are studied, where
\begin{subequations}
{\allowdisplaybreaks
\begin{eqnarray}
f_{\lambda,x,w}^{(-)}(\tau,H) &:=&
\sum_{\gamma \in \overline{Q}^{\vee}} \hspace{-1mm}
e^{2\pi i(\overline{\lambda+\rho}+(K+h^{\vee})\gamma -(K+h^{\vee})w^{-1}x
\, | \, w^{-1}(H+x))} 
\nonumber
\\[2mm]
& &
\times \,\ 
q^{\frac{1}{2(K+h^{\vee})}|\overline{\lambda+\rho}
+(K+h^{\vee})\gamma-(K+h^{\vee})w^{-1}x|^2}
\label{vanish:eqn:2024-728b1}
\\[2mm]
f_{\lambda,x,w}^{(\ast)}(\tau,H) &:=&
q^{-\frac{K+h^{\vee}}{2}|x|^2}
\sum_{\gamma \in \overline{Q}^{\vee}} \hspace{-1mm}
e^{2\pi i(\overline{\lambda+\rho}+(K+h^{\vee})\gamma -(K+h^{\vee})w^{-1}x
\, | \, w^{-1}(H+x+\tau x))} 
\nonumber
\\[2mm]
& &
\times \,\ 
q^{\frac{1}{2(K+h^{\vee})}|\overline{\lambda+\rho}
+(K+h^{\vee})\gamma-(K+h^{\vee})w^{-1}x|^2}
\label{vanish:eqn:2024-728b2}
\end{eqnarray}}
\end{subequations}

In this paper we will show that the RHS of \eqref{vanish:eqn:2024-703e1}
vanishes if the nilpotent element $f$ is of the following form:
\begin{equation}
f \, = \sum_{j=1}^n \, e_{-\beta_j} \hspace{10mm} (n \, \geq \, 2)
\label{vanish:eqn:2024-703f}
\end{equation}
namely, the surviving cases are only $f=e_{-\theta}$ and $f=e_{-\theta_s}$.
In \S \ref{subsec:B2:K=-1:QHR:theta:short}, we will see that 
the quantum Hamiltonian reduction with respect to 
$f=e_{-\theta_s}$ does not necessarily vanish in the case $B_2$. 

\medskip

The author is grateful to Professor Victor Kac for recalling my 
attention to our 2018's work \cite{KW2017d} and fruitful 
collaboration in \cite{KW2024}.

\section{Vanishing of the numerators of quantized characters}
\label{sec:vanishing}


Let us begin this section with noticing a simple fact which 
can be shown very easily.

\medskip

\begin{note}
\label{vanish:note:2024-704a}
Let $\overline{\ggg}$ be a finite-dimensional simple Lie algebra 
of rank $\geq 2$. Then, for $\alpha, \beta \in \overline{\Delta}$ 
such that $\alpha \ne \beta$, \, either 
$(\alpha \, | \, \beta^{\vee}) \, \ne \, 2$ \, or \, 
$(\alpha^{\vee} \, | \, \beta) \, \ne \, 2$.
\end{note}

\medskip

\begin{lemma} 
\label{vanish:lemma:2024-706a}
Let $(x,e,f)$ be an $sl_2$-triple such that $
f=\sum_{j=1}^ne_{-\beta_j} \quad (n \geq 2)$. 
Take $\beta' \in \{\beta_1, \, \cdots \, , \, \beta_n\}$ and 
put $w_0 :=r_{\beta'}$. Then 
\begin{enumerate}
\item[{\rm 1)}] \,\ $w_0x \, = \, w_0^{-1}x \, = \, x-\beta'{}^{\vee}$
\qquad and \quad $w_0H \, = w_0^{-1}H \, = \, H$ \quad 
for ${}^{\forall}H \in \overline{\hhh}^f$.
\item[{\rm 2)}] \,\ $f_{\lambda,x,w_0w}(\tau,H) \, = \, 
f_{\lambda,x,w}(\tau,H)$ 
\quad for ${}^{\forall}w \in \overline{W}$ and 
${}^{\forall}H \in \overline{\hhh}^f$.
\end{enumerate}
\end{lemma}

\begin{proof} 1) is clear. \quad 
2) \, By \eqref{vanish:eqn:2024-703e3}, one has
{\allowdisplaybreaks
\begin{eqnarray*}
& &
f_{\lambda,x,w_0w}(\tau,H) 
\\[2mm]
&=&
\sum_{\gamma \in \overline{Q}^{\vee}}
e^{2\pi i(\overline{\lambda+\rho}+(K+h^{\vee})\gamma
-(K+h^{\vee})w^{-1}w_0^{-1}x \, | \, w^{-1}w_0^{-1}H)} \, 
q^{\frac{1}{2(K+h^{\vee})}|\overline{\lambda+\rho}
+(K+h^{\vee})\gamma-(K+h^{\vee})w^{-1}w_0^{-1}x|^2}
\\[2mm]
&=&
\sum_{\gamma \in \overline{Q}^{\vee}}
e^{2\pi i(\overline{\lambda+\rho}+(K+h^{\vee})\gamma
-(K+h^{\vee})w^{-1}(x-\beta'{}^{\vee}) \, | \, w^{-1}H)} \, 
q^{\frac{1}{2(K+h^{\vee})}|\overline{\lambda+\rho}
+(K+h^{\vee})\gamma-(K+h^{\vee})w^{-1}(x-\beta'{}^{\vee})|^2}
\end{eqnarray*}}
Putting $\gamma' := \gamma+w^{-1}\beta'{}^{\vee}$, this equation 
is rewritten as follows:
{\allowdisplaybreaks
\begin{eqnarray*}
&=&
\sum_{\gamma' \in \overline{Q}^{\vee}}
e^{2\pi i(\overline{\lambda+\rho}+(K+h^{\vee})\gamma'
-(K+h^{\vee})w^{-1}x \, | \, w^{-1}H)} \, 
q^{\frac{1}{2(K+h^{\vee})}|\overline{\lambda+\rho}
+(K+h^{\vee})\gamma'-(K+h^{\vee})w^{-1}x|^2}
\\[2mm]
&=&
f_{\lambda,x,w}(\tau,H) \,\ ,
\end{eqnarray*}}
proving 2).
\end{proof}

\medskip

\begin{thm} 
\label{vanish:thm:2024-801a}
Let $\overline{\ggg}$ be a simple Lie algebra of rank $\geq 2$.
Assume that $(\lambda, \alpha)$ is an element in 
$\hhh^{\ast} \times \overline\Delta_+$ satisfying the condition 
\eqref{vanish:eqn:2024-703b} 
and $(x,e,f=\sum_{j=1}^ne_{-\beta_j})$ is an $sl_2$-triple 
such that $n \geq 2$. Then
$$
A^{[\alpha]}_{\lambda+\rho}
\Big(\tau, \,\ -\tau x+H, \,\ \frac{\tau}{2} \, |x|^2\Big) \,\ = \,\ 0
$$
\end{thm}

\begin{proof} Choose $\beta, \beta' \in \{\beta_1, \, \cdots \, , \, \beta_n\}$
such that $\beta \ne \beta'$ and $(\beta^{\vee}|\beta') \ne 2$, and 
put $w_0 :=r_{\beta'}$. Then, replacing $w$ with $w_0w$ in 
\eqref{vanish:eqn:2024-703e1}
and using Lemma \ref{vanish:lemma:2024-706a}
and noticing that $w_0\beta^{\vee} \, = \, 
w_0^{-1}\beta^{\vee} \, = \, 
\beta^{\vee}-(\beta^{\vee}| \beta') \beta'{}^{\vee}$, one has 
{\allowdisplaybreaks
\begin{eqnarray*}
& & \hspace{-12mm}
A^{[\alpha]}_{\lambda+\rho}
\Big(\tau, \,\ -\tau x+H, \,\ \frac{\tau}{2} \, |x|^2\Big)
\,\ = \,\ 
\frac14 \,  \sum_{w \in \overline{W}}
\varepsilon(w_0w) \, (\beta^{\vee}| w_0w\alpha) \, 
f_{\lambda,x,w_0w}(\tau,H)
\\[2mm]
&=&- \, 
\frac14 \, \sum_{w \in \overline{W}}
\varepsilon(w) \, (w_0^{-1}\beta^{\vee}| w\alpha) \, 
f_{\lambda,x,w}(\tau,H)
\\[2mm]
&=&- \, 
\frac14 \, \sum_{w \in \overline{W}}
\varepsilon(w) \, (\beta^{\vee}| w\alpha) \, 
f_{\lambda,x,w}(\tau,H)
+ \, 
(\beta^{\vee}|\beta') \ 
\underbrace{\frac14 \, \sum_{w \in \overline{W}}
\varepsilon(w) \, (\beta'{}^{\vee}| w\alpha) \, 
f_{\lambda,x,w}(\tau,H)}_{\substack{|| \\[-1mm] 
{\displaystyle \hspace{5mm}
A^{[\alpha]}_{\lambda+\rho}
(\tau, \,\ -\tau x+H, \,\ \tfrac{\tau}{2} \, |x|^2)
}}}
\\[-5mm]
&=&
\big[-1+(\beta^{\vee}|\beta')\big] \, A^{[\alpha]}_{\lambda+\rho}
\Big(\tau, \,\ -\tau x+H, \,\ \frac{\tau}{2} \, |x|^2\Big)
\end{eqnarray*}}
So one has
$$
\big[
\underbrace{2-(\beta^{\vee}|\beta')}_{\substack{
\rotatebox{-90}{$\ne$} \\[0mm] {\displaystyle 0
}}}\big] \, A^{[\alpha]}_{\lambda+\rho}
\Big(\tau, \,\ -\tau x+H, \,\ \frac{\tau}{2} \, |x|^2\Big) \, = \, 0 \, ,
$$
proving Theorem \ref{vanish:thm:2024-801a}.
\end{proof}

\medskip

The calculation in \S  \ref{subsec:B2:K=-1:QHR:theta:short}
shows that 
$A^{[\alpha]}_{\lambda+\rho}(\tau, -\tau x+H, \frac{\tau}{2}|x|^2)$ 
for $\alpha=\theta_s$ does not necessarily vanish.
From  Theorem 4.1 in \cite{KW2017d} and the above Theorem 
\ref{vanish:thm:2024-801a}, we immediately obtain the 
following:

\medskip

\begin{cor} 
\label{vanish:cor:2024-801a}
Let $\overline{\ggg}$ be a simple Lie algebra of rank $\geq 2$.
Assume that $(\lambda, \alpha)$ is an element in 
$\hhh^{\ast} \times \overline\Delta_+$ satisfying the conditions 
{\rm (i)} $\sim$ {\rm (iv)} in Theorem 4.1 in \cite{KW2017d}
and $(x,e,f=\sum_{j=1}^ne_{-\beta_j})$ is an $sl_2$-triple 
such that $n \geq 2$. Then the quantum Hamiltonian reduction 
$H(\lambda)$ of the $\ggg$-module $L(\lambda)$ with respect to 
this $sl_2$-triple vanishes.
\end{cor}

\section{Example $\sim$ $B_2$ and $K=-1$}
\label{sec:ex:B2}

We consider the simple Lie algebra $B_2$ with the 
Dynkin diagram \hspace{-5mm}
\setlength{\unitlength}{1mm}
\begin{picture}(21,9)
\put(6,1){\circle{3}}
\put(17,1){\circle{3}}
\put(7.5,1.5){\vector(1,0){8.1}}
\put(7.5,0.5){\vector(1,0){8.1}}
\put(6,5){\makebox(0,0){$\alpha_1$}}
\put(17,5){\makebox(0,0){$\alpha_2$}}
\end{picture} and inner product 
$\big((\alpha_i|\alpha_j)\big)_{i,j=1,2}
= \begin{pmatrix}
2 & -1 \\
-1 & 1
\end{pmatrix} $. 
The Weyl vector of $B_2$ is 
$\overline{\rho} \, = \, \frac32 \, \alpha_1+2\alpha_2$
with its square length $|\rho|^2 \, = \, \frac52$ and the dual 
Coxeter number $h^{\vee}=3$.
The highest root is $\theta=\alpha_1+2\alpha_2$ 
and the highest short root is $\theta_s=\alpha_1+\alpha_2$. 
The coroot lattice $\overline{Q}^{\vee}
=\zzz \alpha_1^{\vee}+\zzz \alpha_2^{\vee}$ is written as 
{\allowdisplaybreaks
\begin{eqnarray*}
\overline{Q}^{\vee} &=& \big\{
j \alpha_1+k(\alpha_1+2\alpha_2) \quad ; \quad j, \, k \, \in \, \zzz
\big\}
\\[2mm]
&=&
\big\{j (\alpha_1+\alpha_2)+k\alpha_2 \quad ; \quad 
j, \, k \, \in \, \zzz, \quad j+k \, \in \, 2 \, \zzz\big\} \, .
\end{eqnarray*}}
The fundamental integral forms $\overline{\Lambda}_i$ \,\ 
$(i=1, 2)$ defined by 
$(\overline{\Lambda}_i|\alpha_j^{\vee})=\delta_{i,j}$
are related to the simple roots as follows:
$$
\left\{
\begin{array}{lcr}
\alpha_1 &=& 2 \, \overline{\Lambda}_1-2 \, \overline{\Lambda}_2 
\\[1.5mm]
\alpha_2 &=& - \, \overline{\Lambda}_1+2 \, \overline{\Lambda}_2
\end{array}\right. \hspace{20mm} \left\{
\begin{array}{lcc}
\overline{\Lambda}_1 &=& \alpha_1+\alpha_2 
\\[1.5mm]
\overline{\Lambda}_2 &=& \frac12 \, (\alpha_1+2\alpha_2)
\end{array}\right. 
$$
As usual, for each $\alpha \in \overline{\Delta}$, we choose 
an element $e_{\alpha}$ in the root space $\overline{\ggg}_{\alpha}$ 
such that $(e_{\alpha}|e_{-\alpha})=1$ namely
$[e_{\alpha}. \, e_{-\alpha}]=\alpha$.

\subsection{Quantum reduction w.r.to $f=e_{-\theta}$}
\label{subsec:B2:K=-1:QHR:theta}


In this section, we consider the quantum Hamiltonian reduction 
with respect to the $sl_2$-triple $(x,e,f)=(\frac12 \theta, \, 
\frac12 e_{\theta}, \, e_{-\theta})$  for which 
$\overline{\hhh}^f \,\ = \,\ \ccc \, \alpha_1$.

\medskip

First we consider the weight $\lambda = - \Lambda_0$, 
for which $\lambda+\rho=2\Lambda_0+\overline{\rho}$ and 
$(\lambda+\rho| \, \delta-\theta)=0$. Then, by computation 
letting $\alpha=\beta=\theta$ in \eqref{vanish:eqn:2024-703e1} 
and \eqref{vanish:eqn:2024-728a1} 
and \eqref{vanish:eqn:2024-728a2}, we obtain the following:
\\ %
%
\begin{subequations}
{\allowdisplaybreaks
\begin{eqnarray}
A^{[\theta]}_{-\Lambda_0+\rho}
\Big(\tau, \,\ -\tau x+z\alpha_1, \,\ \frac{\tau}{4}\Big)
&=&  
\dfrac{-i}{2} \, [\vartheta_{00}+\vartheta_{01}](\tau,0) \cdot 
\vartheta_{11}(\tau, 2z)
\label{vanish:2024-801a1}
\\[2mm]
A^{[\theta]\, (-)}_{-\Lambda_0+\rho}
\Big(\tau, \,\ -\tau x+z\alpha_1, \,\ \frac{\tau}{4}\Big)
&=& 
\dfrac{-i}{2} \, [\vartheta_{00}+\vartheta_{01}](\tau,0) \cdot 
\vartheta_{11}(\tau, 2z)
\label{vanish:2024-801a2}
\\[2mm]
A^{[\theta] \, (\ast)}_{-\Lambda_0+\rho}
\Big(\tau, \,\ -\tau x+z\alpha_1, \,\ \frac{\tau}{4}\Big)
&=&
\dfrac{-i}{2} \, [\vartheta_{00}-\vartheta_{01}](\tau,0) \cdot 
\vartheta_{11}(\tau, 2z)
\label{vanish:2024-801a3}
\end{eqnarray}}
\end{subequations}
where $\vartheta_{ab}(\tau,z)$ are the Mumford's theta functions 
(\cite{Mum}).

\medskip

Next we consider the weight 
$\lambda \, = \, - \, \Lambda_0+\overline{\Lambda}_2$, 
for which 
$\lambda+\rho=2\Lambda_0+\overline{\Lambda}_1+2\overline{\Lambda}_2$ 
and 
$(\lambda+\rho| \, \delta-(\alpha_1+\alpha_2))=0$. Then, 
by computation letting $\alpha=\alpha_1+\alpha_2$ and 
$\beta=\theta$ in \eqref{vanish:eqn:2024-703e1} 
and \eqref{vanish:eqn:2024-728a1} 
and \eqref{vanish:eqn:2024-728a2}, we obtain the following:
%
{\allowdisplaybreaks
\begin{eqnarray}
A^{[\alpha_1+\alpha_2]}_{- \, \Lambda_0+\overline{\Lambda}_2+\rho}
(\tau, \,\ -\tau x+z\alpha_1, \,\ \tfrac{\tau}{4})
&=&  - \, i \, 
\vartheta_{10}(\tau, 0) \cdot \vartheta_{11}(\tau, 2z)
\label{vanish:eqn:2024-801b}
\\[2mm]
A^{[\alpha_1+\alpha_2] \, (-)}_{- \, \Lambda_0+\overline{\Lambda}_2+\rho}
(\tau, \,\ -\tau x+z\alpha_1, \,\ \tfrac{\tau}{4})
&=& \hspace{5mm}
A^{[\alpha_1+\alpha_2] \, (\ast)}_{- \, \Lambda_0+\overline{\Lambda}_2+\rho}
(\tau, \,\ -\tau x+z\alpha_1, \,\ \tfrac{\tau}{4})
\nonumber
\\[2mm]
&=&
- \,\ 
A^{[\alpha_1+\alpha_2]}_{- \, \Lambda_0+\overline{\Lambda}_2+\rho}
(\tau, \,\ -\tau x+z\alpha_1, \,\ \tfrac{\tau}{4})
\nonumber
\end{eqnarray}}

The modular transformation of these functions are easily calculated 
by the transformation properties of $\vartheta_{ab}$ as follows:

\medskip 

\begin{note} 
\label{vanish:note:2024-801a}
Define the functions $\psi_i^{(1)}(\tau,z) \,\ (i=1,2,3)$ by
$$\left\{
\begin{array}{lcr}
\psi_1^{(1)}(\tau,z) &:=& A^{[\theta]}_{-\Lambda_0+\rho}
(\tau, \, -\frac{\tau}{2} \theta+z\alpha_1, \, \frac{\tau}{4}) \\[3mm]
\psi_2^{(1)}(\tau,z) &:=& A^{[\theta] (\ast)}_{-\Lambda_0+\rho}
(\tau, \, -\frac{\tau}{2} \theta+z\alpha_1, \, \frac{\tau}{4}) \\[3mm]
\psi_3^{(1)}(\tau,z) &:=& 
A^{[\alpha_1+\alpha_2]}_{-\Lambda_0+\overline{\Lambda}_2+\rho}
(\tau, \, -\frac{\tau}{2} \theta+z\alpha_1, \, \frac{\tau}{4})
\end{array}\right.
$$
Then
\begin{enumerate}
\item[{\rm 1)}]
\begin{enumerate}
\item[{\rm (i)}] \quad $\psi_1^{(1)}(\tau,z)+\psi_2^{(1)}(\tau,z) 
\,\ = \,\ 
- \, i  \,\ \vartheta_{00}(\tau,0) \cdot \vartheta_{11}(\tau,2z)$
\item[{\rm (ii)}] \quad $\psi_1^{(1)}(\tau,z)-\psi_2^{(1)}(\tau,z) 
\,\ = \,\ 
- \, i  \,\ \vartheta_{01}(\tau,0) \cdot \vartheta_{11}(\tau,2z)$
\item[{\rm (iii)}] \quad $\psi_3^{(1)}(\tau,z) \hspace{20mm}
\,\ = \,\ 
- \, i  \,\ \vartheta_{10}(\tau,0) \cdot \vartheta_{11}(\tau,2z)$
\end{enumerate}
\item[{\rm 2)}] 
\begin{enumerate}
\item[{\rm (i)}] \quad $[\psi_1^{(1)}+\psi_2^{(1)}]
(-\frac{1}{\tau}, \frac{z}{\tau}) 
\,\ = \,\ 
- \, \tau \, e^{\frac{4\pi iz^2}{\tau}} \, 
[\psi_1^{(1)}+\psi_2^{(1)}](\tau,z)$
\item[{\rm (ii)}] \quad $[\psi_1^{(1)}-\psi_2^{(1)}]
(-\frac{1}{\tau}, \frac{z}{\tau})
\,\ = \,\ 
- \, \tau \, e^{\frac{4\pi iz^2}{\tau}} \, \psi_3^{(1)}(\tau,z)$
\item[{\rm (iii)}] \quad $\psi_3^{(1)}
(-\frac{1}{\tau}, \frac{z}{\tau}) \hspace{13.5mm}
\,\ = \,\ 
- \, \tau \, e^{\frac{4\pi iz^2}{\tau}} \, 
[\psi_1^{(1)}-\psi_2^{(1)}](\tau,z)$
\end{enumerate}
\item[{\rm 3)}]
\begin{enumerate}
\item[{\rm (i)}] \quad $[\psi_1^{(1)}+\psi_2^{(1)}](\tau+1,z)
\,\ = \,\ 
e^{\frac{\pi i}{4}} \, [\psi_1^{(1)}-\psi_2^{(1)}](\tau,z)$
\item[{\rm (ii)}] \quad $[\psi_1^{(1)}-\psi_2^{(1)}](\tau+1,z)
\,\ = \,\ 
e^{\frac{\pi i}{4}} \, [\psi_1^{(1)}+\psi_2^{(1)}](\tau,z)$
\item[{\rm (iii)}] \quad $\psi_3^{(1)}(\tau+1,z) \hspace{13.5mm}
\,\ = \,\ 
i \, \psi_3^{(1)}(\tau,z)$
\end{enumerate}
\end{enumerate}
\end{note}

\medskip

\begin{rem}
\label{vanish:rem:2024-801a}
The modular transformation properties of the functions 
$\psi_i^{(1)}$'s in the above Note \ref{vanish:note:2024-801a} 
are quite different from the ones of 
$A_{\lambda+\rho}^{[\alpha]}$'s in \cite{KW2024}.
This is because the condition $(\lambda+\rho| \delta-\alpha)=0$ 
is not assumed in the calculation of 
the modular transformation formulas in \cite{KW2024}.
Also these modular transformation properties of the numerators 
in Note \ref{vanish:note:2024-801a} are not in consistency with 
those of the denominators given in \cite{KW2024}.
This incoincidence of the modular transformation
properties between numerators and denominators seems to suggest 
that the condition (iv) in Theorem 4.1 in \cite{KW2024} is not
satisfied in this case namely in the case $B_2$ and $K=-1$.
\end{rem}

\subsection{Quantum reduction w.r.to $f=e_{-\theta_s}$}
\label{subsec:B2:K=-1:QHR:theta:short}


In this section, we consider the quantum Hamiltonian reduction 
with respect to the $sl_2$-triple $(x,e,f)=(\alpha_1+\alpha_2, \, 
e_{\alpha_1+\alpha_2}, \, e_{-(\alpha_1+\alpha_2)})$  for which 
$\overline{\hhh}^f \,\ = \,\ \ccc \, \alpha_2$.

\medskip

First we consider the weight $\lambda = - \Lambda_0$, 
for which $\lambda+\rho=2\Lambda_0+\overline{\rho}$ and 
$(\lambda+\rho| \, \delta-\theta)=0$. Then, by computation 
letting $\alpha=\theta$ and $\beta=\alpha_1+\alpha_2=\theta_s$ in 
\eqref{vanish:eqn:2024-703e1} and \eqref{vanish:eqn:2024-728a1} 
and \eqref{vanish:eqn:2024-728a2}, we obtain the following:
\begin{subequations}
{\allowdisplaybreaks
\begin{eqnarray}
A^{[\theta]}_{-\Lambda_0+\rho}
(\tau, \, -\tau \theta_s +z\alpha_2, \, \tfrac{\tau}{2})
&=& \hspace{-2mm}
\sum_{\substack{j, \, k \, \in \, \zzz \\[1mm]
{\rm s.t.} \,\ j+k \, \in \, \zzz_{\rm even} }} \hspace{-5mm}
\big[e^{4\pi i(j+\frac14)z} -e^{-4\pi i(j+\frac14)z}\big] \, 
q^{(j+\frac14)^2 \, + \, (k+\frac14)^2}
\label{vanish:eqn:2024-801c1}
\\[2mm]
A^{[\theta] \, (-)}_{-\Lambda_0+\rho}
(\tau, \, -\tau \theta_s+z\alpha_2, \, \tfrac{\tau}{2})
&=& 
- \, A^{[\theta]}_{\lambda+\rho}
(\tau, \,\ -\tau \theta_s+z\alpha_2, \,\ \tfrac{\tau}{2})
\label{vanish:eqn:2024-801c2}
\\[2mm]
A^{[\theta] \, (\ast)}_{-\Lambda_0+\rho}
(\tau, \, -\tau \theta_s+z\alpha_2, \, \tfrac{\tau}{2})
&=& 
- \hspace{-5mm}
\sum_{\substack{j, \, k \, \in \, \zzz \\[1mm]
{\rm s.t.} \,\ j+k \, \in \, \zzz_{\rm odd} }} \hspace{-5mm}
\big[e^{4\pi i(j+\frac14)z} -e^{-4\pi i(j+\frac14)z}\big] \, 
q^{(j+\frac14)^2 \, + \, (k+\frac14)^2}
\label{vanish:eqn:2024-801c3}
\end{eqnarray}}
\end{subequations}

Next we consider the weight $\lambda = - \Lambda_0+\overline{\Lambda}_2$, 
for which 
$\lambda+\rho=2\Lambda_0+\overline{\Lambda}_2+2\overline{\Lambda}_2$ and 
$(\lambda+\rho| \, \delta-(\alpha_1+\alpha_2))=0$. Then, by computation 
letting $\alpha=\beta=\alpha_1+\alpha_2=\theta_s$ in 
\eqref{vanish:eqn:2024-703e1} and \eqref{vanish:eqn:2024-728a1} 
and \eqref{vanish:eqn:2024-728a2}, we obtain the following:
%
\begin{equation}
A^{[\alpha_1+\alpha_2]}_{- \Lambda_0+\overline{\Lambda}_2+\rho}
(\tau, \,\ -\tau \theta_s+z\alpha_2, \,\ \tfrac{\tau}{2})
\,\ = \,\ 
\theta^{(-)}_{0,1}(\tau,0) \,\ \theta^{(-)}_{1,1}(\tau,2z)
\label{vanish:eqn:2024-801d}
\end{equation}
where \quad $\theta_{j,m}^{(\pm)}(\tau,z) := \sum\limits_{k \in \zzz}
(\pm 1)^k e^{2\pi im(k+\frac{j}{2m})z} q^{m(k+\frac{j}{2m})^2}$
\,\ is the Jacobi's theta function.

\medskip

The modular transformation of these functions are easily calculated 
by the transformation properties of the Jacobi's theta functions 
as follows:

\medskip

\begin{note} 
\label{vanish:note:2024-801b}
Define the functions $\psi_i^{(2)}(\tau,z) \,\ (i=1,2,3)$ by
$$\left\{
\begin{array}{lcr}
\psi_1^{(2)}(\tau,z) &:=& A^{[\theta]}_{-\Lambda_0+\rho}
(\tau, \, -\tau \theta_s+z\alpha_2, \, \frac{\tau}{2}) \\[3mm]
\psi_2^{(2)}(\tau,z) &:=& A^{[\theta] (\ast)}_{-\Lambda_0+\rho}
(\tau, \, -\tau \theta_s+z\alpha_2, \, \frac{\tau}{2}) \\[3mm]
\psi_3^{(2)}(\tau,z) &:=& 
A^{[\alpha_1+\alpha_2]}_{-\Lambda_0+\overline{\Lambda}_2+\rho}
(\tau, \, -\tau \theta_s+z\alpha_2, \, \frac{\tau}{2})
\end{array}\right.
$$
Then
\begin{enumerate}
\item[{\rm 1)}]
\begin{enumerate}
\item[{\rm (i)}] \quad $\psi_1^{(2)}(\tau,z)+\psi_2^{(2)}(\tau,z) 
\,\ = \,\ 
\theta_{\frac12,1}^{(-)}(\tau,0) \, 
\big[\theta_{\frac12,1}^{(-)}-\theta_{-\frac12,1}^{(-)}\big](\tau,2z)$
\item[{\rm (ii)}] \quad $\psi_1^{(2)}(\tau,z)-\psi_2^{(2)}(\tau,z) 
\,\ = \,\ 
\theta_{\frac12,1}^{(+)}(\tau,0) \, 
\big[\theta_{\frac12,1}^{(+)}-\theta_{-\frac12,1}^{(+)}\big](\tau,2z)$
\item[{\rm (iii)}] \quad $\psi_3^{(2)}(\tau,z) \hspace{20mm}
\,\ = \,\ 
\theta_{0,1}^{(-)}(\tau,0) \, \theta_{1,1}^{(-)}(\tau,2z)$
\end{enumerate}
\item[{\rm 2)}] 
\begin{enumerate}
\item[{\rm (i)}] \quad $[\psi_1^{(2)}+\psi_2^{(2)}]
(-\frac{1}{\tau}, \frac{z}{\tau}) 
\,\ = \,\ 
- \, \tau \,  e^{\frac{2\pi iz^2}{\tau}} \, 
[\psi_1^{(2)}+\psi_2^{(2)}](\tau, z)$
\item[{\rm (ii)}] \quad $[\psi_1^{(2)}-\psi_2^{(2)}]
(-\frac{1}{\tau}, \frac{z}{\tau}) 
\,\ = \,\ 
- \, \tau \,  e^{\frac{2\pi iz^2}{\tau}} \, \psi_3^{(2)}(\tau, z)$
\item[{\rm (iii)}] \quad $\psi_3^{(2)}(-\frac{1}{\tau}, \frac{z}{\tau}) 
\hspace{13.5mm}
\,\ = \,\ 
- \, \tau \,  e^{\frac{2\pi iz^2}{\tau}} \, 
[\psi_1^{(2)}-\psi_2^{(2)}](\tau, z)$
\end{enumerate}
\item[{\rm 3)}] 
\begin{enumerate}
\item[{\rm (i)}] \quad $[\psi_1^{(2)}+\psi_2^{(2)}](\tau+1,z)
\,\ = \,\ 
e^{\frac{\pi i}{4}} \, [\psi_1^{(2)}-\psi_2^{(2)}](\tau,z) $
\item[{\rm (ii)}] \quad $[\psi_1^{(2)}-\psi_2^{(2)}](\tau+1,z)
\,\ = \,\ 
e^{\frac{\pi i}{4}} \, [\psi_1^{(2)}+\psi_2^{(2)}](\tau,z) $
\item[{\rm (iii)}] \quad $\psi_3^{(2)}(\tau+1,z) \hspace{13.5mm}
\,\ = \,\ i \, \psi_3^{(2)}(\tau,z) $
\end{enumerate}
\end{enumerate}
\end{note}

\medskip

\begin{rem}
\label{vanish:rem:2024-801b}
The modular transformation properties of $\psi^{(1)}_i$'s 
in Note \ref{vanish:note:2024-801a} are quite similar with 
those of $\psi^{(2)}_i$'s in Note \ref{vanish:note:2024-801b}
except for only difference between the factors 
$e^{\frac{4\pi iz^2}{\tau}}$ and $e^{\frac{2\pi iz^2}{\tau}}$.
\end{rem}

\section{Example $\sim$ $D_4$}
\label{sec:ex:D4}


We consider the simple Lie algebra $D_4$ 
with the Dynkin diagram \hspace{-5mm}
\setlength{\unitlength}{1mm}
\begin{picture}(25,11)
\put(5,11){\circle{3}}
\put(16,11){\circle{3}}
\put(27,11){\circle{3}}
\put(16,0){\circle{3}}
\put(6.5,11){\line(1,0){8}}
\put(17.5,11){\line(1,0){8}}
\put(16,9.5){\line(0,-1){8}}
\put(5,15){\makebox(0,0){$\alpha_1$}}
\put(16,15){\makebox(0,0){$\alpha_2$}}
\put(27,15){\makebox(0,0){$\alpha_3$}}
\put(13,-3){\makebox(0,0){$\alpha_4$}} 
\end{picture} 

\medskip

\noindent
and inner product $\big((\alpha_i|\alpha_j)\big)_{i,j=1,2,3,4} 
=
\begin{pmatrix}
2 & -1 & 0 & 0 \\
-1 & 2 & -1 & -1 \\
0 & -1 & 2 & 0 \\
0 & -1 & 0 & 2
\end{pmatrix}$.
The highest root is $\theta= \alpha_1+2\alpha_2+\alpha_3+\alpha_4$,
and the dual Coxeter number is $h^{\vee}=6$. 
The central charge of the quantum Hamiltonian reduction of a highest 
weight $\widehat{D}_4$-module of level $K$ with respect to the 
$sl_2$-triple $(x,e,f)$, where $f=e_{-\theta}$ and $x=\frac12 \theta$, 
is given by 
\begin{equation}
c(K) \, = \, - \, \dfrac{6(K-1)(K+2)}{K+6} \, .
\label{vanish:eqn:2024-803a}
\end{equation}
We define the coordinates in the Cartan subalgebra $\hhh$ of $\widehat{D}_4$
by
\begin{equation}
2\pi i\Big(-\tau\Lambda_0 
\, + \, z_1 \, \frac{\alpha_1}{2}
\, + \, z_2 \, \frac{\theta}{2}
\, + \, z_3 \, \frac{\alpha_3}{2}
\, + \, z_4 \, \frac{\alpha_4}{2}
\, + \, t\delta\Big) \,\ = \,\ (\tau, z_1,z_2,z_3, z_4,t)
\label{vanish:eqn:2024-803b}
\end{equation}
Then the space $\overline{\hhh}^f$ is written as follows:
\begin{equation}
\overline{\hhh}^f \, = \, \Big\{
H \, = \, 
z_1 \, \frac{\alpha_1}{2}
\, + \, z_3 \, \frac{\alpha_3}{2}
\, + \, z_4 \, \frac{\alpha_4}{2} \quad ; \quad 
z_1, \, z_2, \, z_3 \, \in \, \ccc\Big\}
\label{vanish:eqn:2024-803c}
\end{equation}

\subsection{The case $K=-1$}
\label{subsec:D4:K=-1}


In this section, we consider the functions $f_{ijk}^{(\pm)}$ \, 
and \, $g_{ijk}$ \,\ $(i,j,k \, \in \, \{0,1\} )$ \, defined by 
{\allowdisplaybreaks
\begin{eqnarray*}
f_{000}^{(\pm)}(\tau, z_1,z_3,z_4) &:=&
{\rm ch}_{\Lambda_0}^{(1)} \,\ 
{\rm ch}_{\Lambda_0}^{(3)} \,\ 
{\rm ch}_{\Lambda_0}^{(4)} \,\ 
\chi^{(5,3)}_{1,1}
\,\ \pm \,\ 
{\rm ch}_{\Lambda_1}^{(1)} \,\ 
{\rm ch}_{\Lambda_1}^{(3)} \,\ 
{\rm ch}_{\Lambda_1}^{(4)} \,\ 
\chi^{(5,3)}_{2,1}
\\[2mm]
f_{100}^{(\pm)}(\tau, z_1,z_3,z_4) &:=&
{\rm ch}_{\Lambda_0}^{(1)} \,\ 
{\rm ch}_{\Lambda_1}^{(3)} \,\ 
{\rm ch}_{\Lambda_1}^{(4)} \,\ 
\chi^{(5,3)}_{1,1}
\,\ \pm \,\ 
{\rm ch}_{\Lambda_1}^{(1)} \,\ 
{\rm ch}_{\Lambda_0}^{(3)} \,\ 
{\rm ch}_{\Lambda_0}^{(4)} \,\ 
\chi^{(5,3)}_{2,1}
\\[2mm]
f_{010}^{(\pm)}(\tau, z_1,z_3,z_4) &:=&
{\rm ch}_{\Lambda_1}^{(1)} \,\ 
{\rm ch}_{\Lambda_0}^{(3)} \,\ 
{\rm ch}_{\Lambda_1}^{(4)} \,\ 
\chi^{(5,3)}_{1,1}
\,\ \pm \,\ 
{\rm ch}_{\Lambda_0}^{(1)} \,\ 
{\rm ch}_{\Lambda_1}^{(3)} \,\ 
{\rm ch}_{\Lambda_0}^{(4)} \,\ 
\chi^{(5,3)}_{2,1}
\\[2mm]
f_{001}^{(\pm)}(\tau, z_1,z_3,z_4) &:=&
{\rm ch}_{\Lambda_1}^{(1)} \,\ 
{\rm ch}_{\Lambda_1}^{(3)} \,\ 
{\rm ch}_{\Lambda_0}^{(4)} \,\ 
\chi^{(5,3)}_{1,1}
\,\ \pm \,\ 
{\rm ch}_{\Lambda_0}^{(1)} \,\ 
{\rm ch}_{\Lambda_0}^{(3)} \,\ 
{\rm ch}_{\Lambda_1}^{(4)} \,\ 
\chi^{(5,3)}_{2,1}
\\[2mm]
f_{011}^{(\pm)}(\tau, z_1,z_3,z_4) &:=&
{\rm ch}_{\Lambda_1}^{(1)} \,\ 
{\rm ch}_{\Lambda_0}^{(3)} \,\ 
{\rm ch}_{\Lambda_0}^{(4)} \,\ 
\chi^{(5,3)}_{2,3}
\,\ \pm \,\ 
{\rm ch}_{\Lambda_0}^{(1)} \,\ 
{\rm ch}_{\Lambda_1}^{(3)} \,\ 
{\rm ch}_{\Lambda_1}^{(4)} \,\ 
\chi^{(5,3)}_{2,2}
\\[2mm]
f_{101}^{(\pm)}(\tau, z_1,z_3,z_4) &:=&
{\rm ch}_{\Lambda_0}^{(1)} \,\ 
{\rm ch}_{\Lambda_1}^{(3)} \,\ 
{\rm ch}_{\Lambda_0}^{(4)} \,\ 
\chi^{(5,3)}_{2,3}
\,\ \pm \,\ 
{\rm ch}_{\Lambda_1}^{(1)} \,\ 
{\rm ch}_{\Lambda_0}^{(3)} \,\ 
{\rm ch}_{\Lambda_1}^{(4)} \,\ 
\chi^{(5,3)}_{2,2}
\\[2mm]
f_{110}^{(\pm)}(\tau, z_1,z_3,z_4) &:=&
{\rm ch}_{\Lambda_0}^{(1)} \,\ 
{\rm ch}_{\Lambda_0}^{(3)} \,\ 
{\rm ch}_{\Lambda_1}^{(4)} \,\ 
\chi^{(5,3)}_{2,3}
\,\ \pm \,\ 
{\rm ch}_{\Lambda_1}^{(1)} \,\ 
{\rm ch}_{\Lambda_1}^{(3)} \,\ 
{\rm ch}_{\Lambda_0}^{(4)} \,\ 
\chi^{(5,3)}_{2,2}
\\[2mm]
f_{111}^{(\pm)}(\tau, z_1,z_3,z_4) &:=&
{\rm ch}_{\Lambda_1}^{(1)} \,\ 
{\rm ch}_{\Lambda_1}^{(3)} \,\ 
{\rm ch}_{\Lambda_1}^{(4)} \,\ 
\chi^{(5,3)}_{2,3}
\,\ \pm \,\ 
{\rm ch}_{\Lambda_0}^{(1)} \,\ 
{\rm ch}_{\Lambda_0}^{(3)} \,\ 
{\rm ch}_{\Lambda_0}^{(4)} \,\ 
\chi^{(5,3)}_{2,2}
\end{eqnarray*}}
and
{\allowdisplaybreaks
\begin{eqnarray*}
g_{000}(\tau, z_1,z_3,z_4) &:=&
{\rm ch}_{\Lambda_1}^{(1)} \,\ 
{\rm ch}_{\Lambda_1}^{(3)} \,\ 
{\rm ch}_{\Lambda_1}^{(4)} \,\ 
\chi^{(5,3)}_{1,1}
\,\ - \,\ 
{\rm ch}_{\Lambda_0}^{(1)} \,\ 
{\rm ch}_{\Lambda_0}^{(3)} \,\ 
{\rm ch}_{\Lambda_0}^{(4)} \,\ 
\chi^{(5,3)}_{2,1}
\\[2mm]
g_{100}(\tau, z_1,z_3,z_4) &:=&
{\rm ch}_{\Lambda_1}^{(1)} \,\ 
{\rm ch}_{\Lambda_0}^{(3)} \,\ 
{\rm ch}_{\Lambda_0}^{(4)} \,\ 
\chi^{(5,3)}_{1,1}
\,\ - \,\ 
{\rm ch}_{\Lambda_0}^{(1)} \,\ 
{\rm ch}_{\Lambda_1}^{(3)} \,\ 
{\rm ch}_{\Lambda_1}^{(4)} \,\ 
\chi^{(5,3)}_{2,1}
\\[2mm]
g_{010}(\tau, z_1,z_3,z_4) &:=&
{\rm ch}_{\Lambda_0}^{(1)} \,\ 
{\rm ch}_{\Lambda_1}^{(3)} \,\ 
{\rm ch}_{\Lambda_0}^{(4)} \,\ 
\chi^{(5,3)}_{1,1}
\,\ - \,\ 
{\rm ch}_{\Lambda_1}^{(1)} \,\ 
{\rm ch}_{\Lambda_0}^{(3)} \,\ 
{\rm ch}_{\Lambda_1}^{(4)} \,\ 
\chi^{(5,3)}_{2,1}
\\[2mm]
g_{001}(\tau, z_1,z_3,z_4) &:=&
{\rm ch}_{\Lambda_0}^{(1)} \,\ 
{\rm ch}_{\Lambda_0}^{(3)} \,\ 
{\rm ch}_{\Lambda_1}^{(4)} \,\ 
\chi^{(5,3)}_{1,1}
\,\ - \,\ 
{\rm ch}_{\Lambda_1}^{(1)} \,\ 
{\rm ch}_{\Lambda_1}^{(3)} \,\ 
{\rm ch}_{\Lambda_0}^{(4)} \,\ 
\chi^{(5,3)}_{2,1}
\\[2mm]
g_{011}(\tau, z_1,z_3,z_4) &:=&
{\rm ch}_{\Lambda_0}^{(1)} \,\ 
{\rm ch}_{\Lambda_1}^{(3)} \,\ 
{\rm ch}_{\Lambda_1}^{(4)} \,\ 
\chi^{(5,3)}_{2,3}
\,\ - \,\ 
{\rm ch}_{\Lambda_1}^{(1)} \,\ 
{\rm ch}_{\Lambda_0}^{(3)} \,\ 
{\rm ch}_{\Lambda_0}^{(4)} \,\ 
\chi^{(5,3)}_{2,2}
\\[2mm]
g_{101}(\tau, z_1,z_3,z_4) &:=&
{\rm ch}_{\Lambda_1}^{(1)} \,\ 
{\rm ch}_{\Lambda_0}^{(3)} \,\ 
{\rm ch}_{\Lambda_1}^{(4)} \,\ 
\chi^{(5,3)}_{2,3}
\,\ - \,\ 
{\rm ch}_{\Lambda_0}^{(1)} \,\ 
{\rm ch}_{\Lambda_1}^{(3)} \,\ 
{\rm ch}_{\Lambda_0}^{(4)} \,\ 
\chi^{(5,3)}_{2,2}
\\[2mm]
g_{110}(\tau, z_1,z_3,z_4) &:=&
{\rm ch}_{\Lambda_1}^{(1)} \,\ 
{\rm ch}_{\Lambda_1}^{(3)} \,\ 
{\rm ch}_{\Lambda_0}^{(4)} \,\ 
\chi^{(5,3)}_{2,3}
\,\ - \,\ 
{\rm ch}_{\Lambda_0}^{(1)} \,\ 
{\rm ch}_{\Lambda_0}^{(3)} \,\ 
{\rm ch}_{\Lambda_1}^{(4)} \,\ 
\chi^{(5,3)}_{2,2}
\\[2mm]
g_{111}(\tau, z_1,z_3,z_4) &:=&
{\rm ch}_{\Lambda_0}^{(1)} \,\ 
{\rm ch}_{\Lambda_0}^{(3)} \,\ 
{\rm ch}_{\Lambda_0}^{(4)} \,\ 
\chi^{(5,3)}_{2,3}
\,\ - \,\ 
{\rm ch}_{\Lambda_1}^{(1)} \,\ 
{\rm ch}_{\Lambda_1}^{(3)} \,\ 
{\rm ch}_{\Lambda_1}^{(4)} \,\ 
\chi^{(5,3)}_{2,2}
\end{eqnarray*}}
where \, ${\rm ch}_{\lambda}^{(\ell)}:=
{\rm ch}_{L(\lambda ; \, A^{(1)}_1)}(\tau, z_{\ell})$ \,\ 
$(\ell=1,3,4)$ \, is the normalized character of the irreducible 
$A^{(1)}_1$-module $L(\lambda)$ and $\chi^{(p,q)}_{r,s}$ is 
the normalized character of the irreducible Virasoro module 
with the central charge $z^{(p,q)}= 1-\dfrac{6(p-q)^2}{pq}$
and the vacuum anomaly $h^{(p,q)}_{r,s}=
\dfrac{(rp-sq)^2-(p-q)^2}{4pq}$.
The modular transformation properties of these functions 
are easily computed by using transformation properties 
of $A^{(1)}_1$-characters and Virasoro characters 
given in \cite{KW4} and \cite{W2001} and \cite{W2},
and obtained as follows:

\medskip

\begin{lemma} 
\label{vanish:lemma:2024-801a}
We put \, 
$a := \dfrac{1}{\sqrt{5}} \, \sin \dfrac{2\pi}{5}$, \, 
$b := \dfrac{1}{\sqrt{5}} \, \sin \dfrac{\pi}{5}$ \, and
{\allowdisplaybreaks
\begin{eqnarray*}
& &
J \, := \, \left(
\begin{array}{rrrr}
1 & 1 & 1 & 1 \\[0.5mm]
1 & 1 & -1 & -1 \\[0.5mm]
1 & -1 & 1 & -1 \\[0.5mm]
1 & -1 & -1 & 1 
\end{array} \right) 
\\[2mm]
& &
A \, := \, a \, J \, = \, \left(
\begin{array}{rrrr}
a & a & a & a \\[0.5mm]
a & a & -a & -a \\[0.5mm]
a & -a & a & -a \\[0.5mm]
a & -a & -a & a 
\end{array}\right) \, ,
\hspace{10mm} %
B \, := \, b \, J \, = \, \left(
\begin{array}{rrrr}
b & b & b & b \\[0.5mm]
b & b & -b & -b \\[0.5mm]
b & -b & b & -b \\[0.5mm]
b & -b & -b & b 
\end{array}\right)
\end{eqnarray*}}
Then the $S$-transformation of $f_{ijk}^{(\pm)}$ and $g_{ijk}$ 
are as follows:
{\allowdisplaybreaks
\begin{eqnarray*}
& & \hspace{-5mm}
\left(
\begin{array}{r}
f_{000}^{(+)}\big|_S \\[1mm]
f_{100}^{(+)}\big|_S \\[1mm]
f_{010}^{(+)}\big|_S \\[1mm]
f_{001}^{(+)}\big|_S \\[1mm]
f_{111}^{(+)}\big|_S \\[1mm]
f_{011}^{(+)}\big|_S \\[1mm]
f_{101}^{(+)}\big|_S \\[1mm]
f_{110}^{(+)}\big|_S
\end{array}\right)
\, = \, \left(
\begin{array}{r|r}
A & B \\[1mm]
\hline 
-B & A
\end{array}\right)
\left(
\begin{array}{r}
g_{\, 000} \\[1mm]
g_{\, 100} \\[1mm]
g_{\, 010} \\[1mm]
g_{\, 001} \\[1mm]
g_{\, 111} \\[1mm]
g_{\, 011} \\[1mm]
g_{\, 101}\\[1mm]
g_{\, 110}
\end{array}\right) \, , \hspace{10mm} %
\left(
\begin{array}{r}
g_{\, 000}\big|_S \\[1mm]
g_{\, 100}\big|_S \\[1mm]
g_{\, 010}\big|_S \\[1mm]
g_{\, 001}\big|_S \\[1mm]
g_{\, 111}\big|_S \\[1mm]
g_{\, 011}\big|_S \\[1mm]
g_{\, 101}\big|_S \\[1mm]
g_{\, 110}\big|_S
\end{array}\right)
\, = \, \left(
\begin{array}{r|r}
A & -B \\[1mm]
\hline 
B & A
\end{array}\right)
\left(
\begin{array}{r}
f_{000}^{(+)} \\[1mm]
f_{100}^{(+)} \\[1mm]
f_{010}^{(+)} \\[1mm]
f_{001}^{(+)} \\[1mm]
f_{111}^{(+)} \\[1mm]
f_{011}^{(+)} \\[1mm]
f_{101}^{(+)} \\[1mm]
f_{110}^{(+)}
\end{array}\right)
\\[2mm]
& & \hspace{-5mm}
\left(
\begin{array}{r}
f_{000}^{(-)}\big|_S \\[1mm]
f_{100}^{(-)}\big|_S \\[1mm]
f_{010}^{(-)}\big|_S \\[1mm]
f_{001}^{(-)}\big|_S \\[1mm]
f_{111}^{(-)}\big|_S \\[1mm]
f_{011}^{(-)}\big|_S \\[1mm]
f_{101}^{(-)}\big|_S \\[1mm]
f_{110}^{(-)}\big|_S
\end{array}\right)
\, = \, \left(
\begin{array}{r|r}
A & B \\[1mm]
\hline 
B & -A
\end{array}\right)
\left(
\begin{array}{r}
f_{000}^{(-)} \\[1mm]
f_{100}^{(-)} \\[1mm]
f_{010}^{(-)} \\[1mm]
f_{001}^{(-)} \\[1mm]
f_{111}^{(-)} \\[1mm]
f_{011}^{(-)} \\[1mm]
f_{101}^{(-)} \\[1mm]
f_{110}^{(-)}
\end{array}\right)
\end{eqnarray*}}
\end{lemma}

\medskip

\begin{lemma}  
\label{vanish:lemma:2024-801b}
The $T$-transformation of $f_{ijk}^{(\pm)}$ and $g_{ijk}$ are 
as follows:
{\allowdisplaybreaks
\begin{eqnarray*}
f^{(\pm)}_{ijk}\big|_T &=& \left\{
\begin{array}{rcl}
e^{-\frac{\pi i}{5}} \, f^{(\mp)}_{000} & & {\rm if} \quad 
(i,j,k) \, = \, (000)
\\[2mm]
e^{\frac{4\pi i}{5}} \, f^{(\mp)}_{ijk} & & {\rm if} \quad 
(i,j,k) \, = \, (100), \, (010), \, (001)
\\[2mm]
e^{\frac{\pi i}{5}} \, f^{(\mp)}_{ijk} & & {\rm if} \quad 
(i,j,k) \, = \, (011), \, (101), \, (110)
\\[2mm]
- \, e^{\frac{\pi i}{5}} \, f^{(\mp)}_{111} & & {\rm if} \quad 
(i,j,k) \, = \, (111)
\end{array}\right.
\\[2mm]
g_{\, ijk}\big|_T &=& \left\{
\begin{array}{rcl}
- \, e^{\frac{3\pi i}{10}} \, g_{\, 000} & & {\rm if} \quad 
(i,j,k) \, = \, (000)
\\[2mm]
e^{\frac{3\pi i}{10}} \, g_{\, ijk} & & {\rm if} \quad 
(i,j,k) \, = \, (100), \, (010), \, (001)
\\[2mm]
e^{\frac{7\pi i}{10}} \, g_{\, ijk} & & {\rm if} \quad 
(i,j,k) \, = \, (011), \, (101), \, (110)
\\[2mm]
- \, e^{\frac{7\pi i}{10}} \, g_{\, 111} & & {\rm if} \quad 
(i,j,k) \, = \, (111)
\end{array}\right.
\end{eqnarray*}}
\end{lemma}

\medskip

The characters of $D_4$-Deligne series of level $K=-1$ are 
written by these functions $f_{ijk}^{(+)}$'s as follows:

\medskip

\begin{prop} \,\ 
\label{vanish:prop:2024-801a}
\begin{enumerate}
\item[{\rm 1)}] \quad $\overset{w}{\rm ch}_{H(-\Lambda_0)}
\,\ = \hspace{3.5mm} f^{(+)}_{000}$
\item[{\rm 2)}] \quad $\overset{w}{\rm ch}_{H(\Lambda)}
\,\ = \,\ \left\{
\begin{array}{lcl}
f^{(+)}_{100} & &{\rm if} \quad \Lambda \, = \, 
-\Lambda_0+\overline{\Lambda}_3+\overline{\Lambda}_4
\\[2mm]
f^{(+)}_{010} & &{\rm if} \quad \Lambda \, = \, 
-\Lambda_0+\overline{\Lambda_1}+\overline{\Lambda_4}
\\[2mm]
f^{(+)}_{001} & &{\rm if} \quad \Lambda \, = \, 
-\Lambda_0+\overline{\Lambda}_1+\overline{\Lambda}_3
\end{array}\right. $
\item[{\rm 3)}] \quad $\overset{w}{\rm ch}_{H(\Lambda)}
\,\ = \,\ \left\{
\begin{array}{lcl}
f^{(+)}_{011} & &{\rm if} \quad \Lambda \, = \, 
-\Lambda_0+\overline{\Lambda}_1
\\[2mm]
f^{(+)}_{101} & &{\rm if} \quad \Lambda \, = \, 
-\Lambda_0+\overline{\Lambda}_3
\\[2mm]
f^{(+)}_{110} & &{\rm if} \quad \Lambda \, = \, 
-\Lambda_0+\overline{\Lambda}_4
\end{array}\right. $
\item[{\rm 4)}] \quad $\overset{w}{\rm ch}_{H(\Lambda)}
\,\ = \hspace{6.5mm} 
f^{(+)}_{111} \hspace{7.7mm} {\rm if} \quad
\Lambda \, = \, -\Lambda_0+\overline{\Lambda}_2$
\end{enumerate}
\end{prop}

\begin{proof}
1) is due to Kawasetsu's results \cite{Kawa1505}. 
For 2) $\sim$ 4), we compute the characters of $L(\Lambda)$'s
by using Theorem \ref{vanish:thm:2024-704a}.
Then, letting $\tau \rightarrow \tau+1$, we get the signed 
characters of $L(\Lambda)$'s. 
For these signed characters we compute their polar parts, 
namely the terms of $q$ with non-positive power,
and their modular transformation explicitly. 
Then, compairing these data with those for $f_{ijk}^{(-)}$'s, 
we conclude that both coincide, namely we obtain the 
relation between signed characters and $f_{ijk}^{(-)}$'s. 
Letting $\tau \rightarrow \tau-1$ in these formulas,
we obtain the relation between characters and $f_{ijk}^{(+)}$'s
as desired. 
\end{proof}

The asymptotic behavior of these characters as $\tau \downarrow 0$ 
is given by the following:

\medskip

\begin{prop} \,\ 
\label{vanish:prop:2024-801b}
\begin{enumerate}
\item[{\rm 1)}] \,\ $(i,j,k) \, = \, (000), \,\ (100), \,\ (010), \,\ (001)
\quad \Longrightarrow \quad 
f^{(+)}_{ijk}(\tau,0,0,0) \, 
\overset{\substack{\tau \, \downarrow \, 0 \\[0.3mm] }}{\sim} \, 
\dfrac{1}{\sqrt{5}} \, \sin \dfrac{\pi}{5} \, 
e^{\frac{3}{10} \cdot \frac{\pi i}{\tau}}$

\item[{\rm 2)}] \,\ $(i,j,k) \, = \, (011), \,\ (101), \,\ (110), \,\ (111)
\quad \Longrightarrow \quad 
f^{(+)}_{ijk}(\tau,0,0,0) \, 
\overset{\substack{\tau \, \downarrow \, 0 \\[0.3mm] }}{\sim} \, 
\dfrac{1}{\sqrt{5}} \, \sin \dfrac{2\pi}{5} \, 
e^{\frac{3}{10} \cdot \frac{\pi i}{\tau}}$
\end{enumerate}
\end{prop}

\begin{proof}
Applying Theorem 4.17 in \cite{W2001} to $L(\Lambda_i, A^{(1)}_1)$,
we have
$$
{\rm ch}_{L(\Lambda_j, A^{(1)}_1)}(\tau,0,0)
\overset{\substack{\tau \, \downarrow \, 0 \\[0.3mm] }}{\sim} 
\dfrac{1}{\sqrt{2}} \, e^{\frac{\pi i}{12\tau}} \hspace{10mm}
(i=0,1)
$$
And, for the asymptotic behavior of the Virasoro character, 
we have 
$$ \hspace{-10mm}
\left\{
\begin{array}{lcc}
\chi^{(5,3)}_{1,1} (\tau) 
& \hspace{-2mm}
\overset{\substack{\tau \, \downarrow \, 0 \\[0.3mm] }}{\sim} 
\hspace{-2mm} & 
\sqrt{\dfrac25} \, 
\sin \dfrac{\pi}{5} \cdot e^{\frac{\pi i}{20 \tau}}
\\[4mm]
\chi^{(5,3)}_{2,1} (\tau) 
& \hspace{-2mm}
\overset{\substack{\tau \, \downarrow \, 0 \\[0.3mm] }}{\sim} 
\hspace{-2mm} & 
\sqrt{\dfrac25} \, 
\sin \dfrac{\pi}{5} \cdot e^{\frac{\pi i}{20 \tau}}
\end{array}\right. \hspace{10mm} 
\left\{
\begin{array}{lcc}
\chi^{(5,3)}_{2,2} (\tau) 
& \hspace{-2mm}
\overset{\substack{\tau \, \downarrow \, 0 \\[0.3mm] }}{\sim} 
\hspace{-2mm} & 
\sqrt{\dfrac25} \, 
\sin \dfrac{2\pi}{5} \cdot e^{\frac{\pi i}{20 \tau}}
\\[4mm]
\chi^{(5,3)}_{2,3} (\tau) 
& \hspace{-2mm}
\overset{\substack{\tau \, \downarrow \, 0 \\[0.3mm] }}{\sim} 
\hspace{-2mm} & 
\sqrt{\dfrac25} \, 
\sin \dfrac{2\pi}{5} \cdot e^{\frac{\pi i}{20 \tau}}
\end{array}\right. 
$$
by Theorem 7.1.5 in \cite{W2}. Using these and definition of 
$f_{ijk}^{(+)}$'s, we obtain the asymptotics of $f_{ijk}^{(+)}$'s
as desired.
\end{proof}

\subsection{The case $K=-2$ $\sim$ 
Denominator identity for $W(\widehat{D}_4, e_{-\theta})$}
\label{subsec:ex:D4:K=-2}

\medskip

By \eqref{vanish:eqn:2024-803a}, one has $c(K)=0$ if $K=-2$, so
one may expect there exists the trivial representation of 
the minimal W-algebra $W(\widehat{D}_4, e_{-\theta})_K$
when $K=-2$, namely one may expect to obtain the denominator identity 
from the quantum Hamiltonian reduction of the 
$\widehat{D}_4$-module $L(-2\Lambda_0)$.
The denominator of the minimal W-algebra $W(\widehat{D}_4, e_{-\theta})$
is given by 
\begin{eqnarray}
\lefteqn{ \hspace{-10mm}
\overset{w}{R}(\tau, H) \,\ = \,\ 
\frac{1}{\eta(\tau)^3} \prod_{i=1,3,4}\vartheta_{11}(\tau, z_i)
\cdot 
\vartheta_{01}\Big(\tau, \frac{z_1+z_3+z_4}{2}\Big)}
\nonumber
\\[3mm]
& \times &
\vartheta_{01}\Big(\tau, \frac{-z_1+z_3+z_4}{2}\Big) \, 
\vartheta_{01}\Big(\tau, \frac{z_1-z_3+z_4}{2}\Big) \, 
\vartheta_{01}\Big(\tau, \frac{z_1+z_3-z_4}{2}\Big)
\label{vanish:eqn:2024-803d}
\end{eqnarray}
where \,\ 
$H = \sum_{i=1,3,4} z_i \, \frac{\alpha_i}{2} \, \in \, \overline{\hhh}^f$.
Then, working out this calculation, we obtain the following:

\medskip



\begin{prop} \,\
\label{prop:Vol.353p.64a}
{\rm (Denominator identity of $W(\widehat{D}_4, e_{-\theta})$)}
\begin{subequations}
\begin{enumerate}
\item[{\rm 1)}] \,\ 

\vspace{-14.5mm}

{\allowdisplaybreaks
\begin{eqnarray}
& & \hspace{-25mm}
\frac{i}{\eta(\tau)^3}
\Big[\prod_{j=1,3,4}\vartheta_{11}(\tau, z_j)\Big] \, 
\vartheta_{01}\Big(\tau, \frac{z_1+z_3+z_4}{2}\Big)
\vartheta_{01}\Big(\tau, \frac{-z_1+z_3+z_4}{2}\Big)
\nonumber
\\[1mm]
& & \hspace{11mm}
\times \,\ 
\vartheta_{01}\Big(\tau, \frac{z_1-z_3+z_4}{2}\Big)
\vartheta_{01}\Big(\tau, \frac{z_1+z_3-z_4}{2}\Big)
\nonumber
\\[4.5mm]
= & &
\sum_{\substack{(m_0,m_1,m_3,m_4) \\[0mm]
\rotatebox{-90}{$\in$} \\[0mm]
\{(4222), \, (1111), \, (1133), \, (3131), \, (3113)\}
}} \hspace{-15mm}
\theta_{m_0,4}(\tau,0) 
\prod_{j=1,3,4}\big[
\theta_{m_j,4}-\theta_{-m_j,4}\big](\tau, z_j)
\nonumber
\\[2mm]
&-&
\sum_{\substack{(m_0,m_1,m_3,m_4) \\[0mm]
\rotatebox{-90}{$\in$} \\[0mm]
\{(0222), \, (3333), \, (3311), \, (1313), \, (1331)\}
}} \hspace{-15mm}
\theta_{m_0,4}(\tau,0) 
\prod_{j=1,3,4}\big[
\theta_{m_j,4}-\theta_{-m_j,4}\big](\tau, z_j)
\label{eqn:2024-518a1}
\end{eqnarray}}
\item[{\rm 2)}] \,\ 

\vspace{-14.5mm}

{\allowdisplaybreaks
\begin{eqnarray}
& & \hspace{-14mm}
\frac{i}{\eta(\tau)^3}
\Big[\prod_{j=1,3,4}\vartheta_{11}(\tau, z_j)\Big] \, 
\vartheta_{00}\Big(\tau, \frac{z_1+z_3+z_4}{2}\Big)
\vartheta_{00}\Big(\tau, \frac{-z_1+z_3+z_4}{2}\Big)
\nonumber
\\[1mm]
& & \hspace{21mm}
\times \,\ 
\vartheta_{00}\Big(\tau, \frac{z_1-z_3+z_4}{2}\Big)
\vartheta_{00}\Big(\tau, \frac{z_1+z_3-z_4}{2}\Big)
\nonumber
\\[4.5mm]
= &- &
\sum_{\substack{(m_0,m_1,m_3,m_4) \\[0mm]
\rotatebox{-90}{$\in$} \\[0mm]
\{(4222), \, (1111), \, (1133), \, (3131), \, (3113)\}
}} \hspace{-15mm}
(-1)^{m_1} \, \theta_{m_0,4}(\tau,0) 
\prod_{j=1,3,4}\big[
\theta_{m_j,4}-\theta_{-m_j,4}\big](\tau, z_j) \hspace{10mm}
\nonumber
\\[2mm]
&+&
\sum_{\substack{(m_0,m_1,m_3,m_4) \\[0mm]
\rotatebox{-90}{$\in$} \\[0mm]
\{(0222), \, (3333), \, (3311), \, (1313), \, (1331)\}
}} \hspace{-15mm}
(-1)^{m_1} \, \theta_{m_0,4}(\tau,0) 
\prod_{j=1,3,4}\big[
\theta_{m_j,4}-\theta_{-m_j,4}\big](\tau, z_j)
\label{eqn:2024-518a2}
\end{eqnarray}}
\item[{\rm 3)}] \,\ 

\vspace{-14.5mm}

{\allowdisplaybreaks
\begin{eqnarray}
& & \hspace{-23mm}
\frac{i}{\eta(\tau)^3}
\Big[\prod_{j=1,3,4}\vartheta_{11}(\tau, z_j)\Big] \, 
\vartheta_{11}\Big(\tau, \frac{z_1+z_3+z_4}{2}\Big)
\vartheta_{11}\Big(\tau, \frac{-z_1+z_3+z_4}{2}\Big)
\nonumber
\\[1mm]
& & \hspace{12.5mm}
\times \,\ 
\vartheta_{11}\Big(\tau, \frac{z_1-z_3+z_4}{2}\Big)
\vartheta_{11}\Big(\tau, \frac{z_1+z_3-z_4}{2}\Big)
\nonumber
\\[2mm]
&=&
- \hspace{-26mm}
\sum_{\substack{(m_0,m_1,m_3,m_4) \\[0mm]
\rotatebox{-90}{$\in$} \\[0mm] \hspace{26mm}
\{(4222), \, (1333), \, (1311), \, (3313), \, (3331)\}
}} \hspace{-30mm}
\theta_{m_0,4}(\tau,0) 
\prod_{j=1,3,4}\big[
\theta_{m_j,4}-\theta_{-m_j,4}\big](\tau, z_j)
\nonumber
\\[2mm]
& & 
+ \hspace{-26mm}
\sum_{\substack{(m_0,m_1,m_3,m_4) \\[0mm]
\rotatebox{-90}{$\in$} \\[0mm] \hspace{26mm}
\{(0222), \, (3111), \, (3133), \, (1131), \, (1113)\}
}} \hspace{-30mm}
\theta_{m_0,4}(\tau,0) 
\prod_{j=1,3,4}\big[
\theta_{m_j,4}-\theta_{-m_j,4}\big](\tau, z_j)
\label{eqn:2024-518a3}
\end{eqnarray}}
\item[{\rm 4)}] \,\ 

\vspace{-14.5mm}

{\allowdisplaybreaks
\begin{eqnarray}
& & \hspace{-23mm}
\frac{i}{\eta(\tau)^3}
\Big[\prod_{j=1,3,4}\vartheta_{11}(\tau, z_j)\Big] \, 
\vartheta_{10}\Big(\tau, \frac{z_1+z_3+z_4}{2}\Big)
\vartheta_{10}\Big(\tau, \frac{-z_1+z_3+z_4}{2}\Big)
\nonumber
\\[1mm]
& & \hspace{12.5mm}
\times \,\ 
\vartheta_{10}\Big(\tau, \frac{z_1-z_3+z_4}{2}\Big)
\vartheta_{10}\Big(\tau, \frac{z_1+z_3-z_4}{2}\Big)
\nonumber
\\[2mm]
&=&
- \hspace{-26mm}
\sum_{\substack{(m_0,m_1,m_3,m_4) \\[0mm]
\rotatebox{-90}{$\in$} \\[0mm] \hspace{26mm}
\{(4222), \, (1333), \, (1311), \, (3313), \, (3331)\}
}} \hspace{-30mm}
(-1)^{m_1} \, \theta_{m_0,4}(\tau,0) 
\prod_{j=1,3,4}\big[
\theta_{m_j,4}-\theta_{-m_j,4}\big](\tau, z_j)
\nonumber
\\[2mm]
& & 
+ \hspace{-26mm}
\sum_{\substack{(m_0,m_1,m_3,m_4) \\[0mm]
\rotatebox{-90}{$\in$} \\[0mm] \hspace{26mm}
\{(0222), \, (3111), \, (3133), \, (1131), \, (1113)\}
}} \hspace{-30mm}
(-1)^{m_1} \, \theta_{m_0,4}(\tau,0) 
\prod_{j=1,3,4}\big[
\theta_{m_j,4}-\theta_{-m_j,4}\big](\tau, z_j)
\label{eqn:2024-518a4}
\end{eqnarray}}
\end{enumerate}
\end{subequations}
\end{prop}

\begin{proof} We consider $\Lambda=-2\Lambda_0$, which satisfies 
$(\Lambda+\rho|\delta-\alpha)=0$ where $\alpha=
\sum_{i=1}^4\alpha_i$. Then, after very long calculations 
by letting $\Lambda:=-2\Lambda_0$ and $\beta:=\theta$
in \eqref{vanish:eqn:2024-703e1} and \eqref{vanish:eqn:2024-703e2},
we obtain the formula \eqref{eqn:2024-518a1}. 
But we do not know whether $\Lambda=-2\Lambda_0$ satisfies the 
condition (iv) in Theorem 4.1 in \cite{KW2017d}, 
namely whether the RHS of \eqref{eqn:2024-518a1} gives the 
right numerator of the quantum reduction of $L(-2\Lambda_0)$.
So, at this stage, the formula \eqref{eqn:2024-518a1} is a 
\lq \lq candidate" of the denominator identity.
In order to confirm that \eqref{eqn:2024-518a1} is a right formula, 
we multiply $1/\eta(\tau)^4$ to both sides of \eqref{eqn:2024-518a1}
and compute and compare their polar parts and modular transformations.
Then we see that they just coincide and the formula \eqref{eqn:2024-518a1}
is established.  2) follows from 1) by letting $z_i \rightarrow z_i+1$, and 
3) follows from 1) by letting $z_i \rightarrow z_i+\tau$, and 
4) follows from 2) by letting $z_i \rightarrow z_i+\tau$.
\end{proof}

\end{document}